\documentclass[a4paper]{article}
\usepackage{blindtext}
\usepackage[utf8]{inputenc}
\usepackage{amsmath}
\usepackage{amsthm}
\usepackage{amssymb,stmaryrd}
\usepackage{csquotes}

\usepackage[toc,page]{appendix}
\usepackage[explicit]{titlesec}

\usepackage[
backend=biber,
style=alphabetic,
sorting=nyt,
giveninits=true
]{biblatex}
\usepackage{graphicx}
\usepackage[english]{babel}
\usepackage{amsfonts}
\usepackage{tikz-cd}
\usepackage{tikz}
\usepackage{tikz}
\usetikzlibrary{matrix}
\usepackage[T1]{fontenc}
\usepackage{float}
\usepackage{graphicx}
\usepackage{caption}
\usepackage{subcaption}
\usepackage{hyperref}
\usepackage{diagbox}
\usepackage{url}
\usepackage{mathrsfs}
\usepackage{bm}
\hypersetup{
unicode=true
}
\usepackage{geometry}

\usepackage[inline]{enumitem}

\newcommand*{\QEDB}{\hfill\ensuremath{\square}}
\makeatletter
\newcommand{\colim@}[2]{%
  \vtop{\m@th\ialign{##\cr
    \hfil$#1\operator@font colim$\hfil\cr
    \noalign{\nointerlineskip\kern1.5\ex@}#2\cr
    \noalign{\nointerlineskip\kern-\ex@}\cr}}%
}
\newcommand{\colim}{%
  \mathop{\mathpalette\colim@{\rightarrowfill@\textstyle}}\nmlimits@
}
\newcommand{\prlim@}[2]{%
  \vtop{\m@th\ialign{##\cr
    \hfil$#1\operator@font lim$\hfil\cr
    \noalign{\nointerlineskip\kern1.5\ex@}#2\cr
    \noalign{\nointerlineskip\kern-\ex@}\cr}}%
}
\newcommand{\prlim}{%
  \mathop{\mathpalette\prlim@{\leftarrowfill@\textstyle}}\nmlimits@
}

\newcommand{\DMO}[1]{\DeclareMathOperator{#1}}
\newcommand{\wh}[1]{\widehat{#1}}
\DMO{\lar}{\longrightarrow}
\DMO{\lal}{\longleftarrow}
\DMO{\Spf}{Spf}
\DMO{\Mat}{Mat}
\DMO{\ars}{\twoheadrightarrow}
\DMO{\als}{\twoheadleftarrow}
\DMO{\ari}{\hookrightarrow}
\DMO{\ali}{\hookleftarrow}
\newcommand{\cosimpl}[1]{\text{cs}{#1}}
\newcommand{\simpl}[1]{\text{s}{#1}}
\makeatother
\usepackage[inline]{enumitem}

\numberwithin{equation}{subsection}
\newtheorem*{thrm*}{Theorem}
\newtheorem{thrm}[equation]{Theorem}
\newtheorem*{prop*}{Proposition}
\newtheorem{prop}[equation]{Proposition}
\newtheorem{lemma}[equation]{Lemma}
\newtheorem{theorem}[equation]{Theorem}
\newtheorem{cor}[equation]{Corollary}

\theoremstyle{definition}
\newtheorem*{dfn*}{Definition}
\newtheorem{dfn}[equation]{Definition}

\newtheorem{notation}[equation]{Notation}

\newtheorem{rmrk}[equation]{Remark}

\DMO{\op}{op}
\DMO{\Res}{Res}
\DMO{\coker}{CoKer}
\DMO{\Sym}{Sym}
\newcommand{\Mod}{\mathrm{Mod}}
\DMO{\id}{id}
\DMO{\cdgAlg}{cdg-Alg}
\DMO{\dgAlg}{dg-Alg}
\DMO{\dgMod}{dg-Mod}
\newcommand{\Top}{ {\mathrm{Top}} }
\DMO{\Fun}{Fun}
\DMO{\Bin}{Bin}
\DMO{\IBin}{IBin}
\DMO{\IGamma}{I\Gamma}
\DMO{\SP}{SP}
\DMO{\Sets}{Sets}
\DMO{\Tot}{Tot}
\DMO{\Barc}{Bar}
\DMO{\sSet}{sSet}
\DMO{\Ab}{Ab}
\DMO{\Hom}{Hom}
\DMO{\Alg}{Alg}
\DMO{\BiAlg}{BiAlg}
\DMO{\MooreSp}{M}
\DMO{\coMooreSp}{cM}
\DMO{\aN}{N}
\DMO{\aC}{C}
\DMO{\Ho}{Ho}
\DMO{\lspan}{span}

\DMO{\im}{Im}
\DMO{\Tors}{Tors}
\DMO{\aK}{K}
\DMO{\Fin}{Fin}
\DMO{\res}{res}
\DMO{\Ch}{Ch}
\DMO{\Coker}{CoKer}
\DMO{\Ker}{Ker}
\DMO{\Sk}{Sk}
\DMO{\Ext}{Ext}
\DMO{\EM}{EM}
\DMO{\pt}{pt}
\DMO{\Spec}{Spec}
\DMO{\hht}{{\hat{}}}
\DMO{\Maps}{Maps}
\DMO{\Set}{Set}
\DMO{\Gr}{Gr}
\DMO{\Tor}{Tor}
\DMO{\AW}{AW}
\DMO{\EZ}{EZ}
\DMO{\cof}{Cone}
\DMO{\DP}{DPT}
\newcommand\mb[1]
	{\mathbb{#1}}
\newcommand\os[2]
	{\overset{#1}{#2}}

\newcommand\ul[1]
	{\underline{#1}}
\newcommand\ol[1]
	{\overline{#1}}
\newcommand\mc[1]
	{\mathcal{#1}}

\newcommand{\ar}[1][]
{\xrightarrow{#1}}

\DMO{\zgqis}{\os{\sim}{\leftrightsquigarrow}}
\DMO{\kk}{\Bbbk}

\title{On Equivalences of Derived Exponential Functors}
\author{Alexander Zakharov
\footnote{
Faculty of Mathematics, Higher School of Economics, 6 Usacheva ulitsa, Moscow, Russia 119048;\\
Chennai Mathematical Institute, H1, SIPCOT IT Park, Siruseri, Kelambakkam, India, 603103;\\
email: xaxa3217 ( •\_•) gmail  com}
}
\date{}
\begin{document}
\maketitle
\begin{abstract}
	A strong symmetric monoidal functor 
	$F\colon (\Ab^{free,fg},\oplus)\ar (\Mod(\kk)^{flat},\otimes)$
	is determined by the Hopf algebra $F(\mb{Z})$ over the ring $\kk$.
	We will show that
	the algebra structure on the left derived functor 
	$\mb{L}^* F(P)$ can be recovered from 
	the augmented coalgebra structure on 
	$F(\mb{Z})$ for $P\in D^{\ge 3}_{perf}(\Ab)$.
	Using a similar technique 
	we will prove that the multiplicative Dold-Puppe-Thom isomorphism 
	$H_*(K(A,n);\mb{Z})\simeq \mb{L}_*\Sym A[n]$ 
	is functorial in $A\in \Ab^{fg}$ 
	whenever $n\ge 2$. By contrast, if $n<1$, this is known to be false in general.
\end{abstract}
\tableofcontents

\newpage

\section{Introduction}
Fix a base ring $\kk$ and assume
$F\colon (\Ab^{\mathrm{free,fg}},\oplus)\ar (\Mod(\kk)^{\mathrm{flat}},\otimes)$ is a strong symmetric monoidal functor 
from the category of finitely generated free abelian groups $\Ab^{free,fg}$ to the category 
of flat $\kk$-modules.
Following \cite{Touze}, we call such $F$ an \emph{exponential functor}.
The maps 
$$+\colon \mb{Z}\oplus \mb{Z}\ar \mb{Z},\ \Delta\colon \mb{Z}\ar \mb{Z}\oplus \mb{Z},\ 0\ar \mb{Z},\ \mb{Z}\ar 0,$$
provide $F(\mb{Z})$ with a commutative cocommutative Hopf algebra 
structure over $\kk$. Conversely, any co/commutative Hopf algebra gives a symmetric strong monoidal functor as above.

In the following we will use the cohomological grading notation with the default coefficients ring $\mb{Z}$.
Let $D^{<-1}_{perf}(\Ab)$ denotes the full subcategory of bounded complexes in $D(\Ab)$ with finitely generated cohomology concentrated
in degress $<-1$.
Denote by $D^{<-1}_{perf}(\Ab)^\vee\subset D^{>0}_{perf}(\Ab)$ the image of $D^{<-1}_{perf}(\Ab)$ under the functor $\mc{R}Hom(-,\mb{Z})$.

Since $P\in D^{>0}_{perf}(\Ab)$ can be represented by a complex of free modules in $\Ch^{\ge 0}(\Ab)$,
we can evaluate the \emph{left derived functor} of $F$ 
by the formula 
\begin{equation}\label{eq:derfun}
	\mb{L}F(P):=N_{\kk}(F(K_{\mb{Z}}(P)))\in D(\Mod(\kk)),
\end{equation} where 
we are using 
the usual Dold-Kan equivalence of the cosimplicial $\kk$-modules and the cochain complexes in non-negative degrees:
$N_{\kk}\colon \cosimpl{\Mod(\kk)} \substack{\lar\\\lal} \Ch^{\ge 0}(\Mod(\kk))\colon K_{\kk}$.

Let us recall that in the cosimplicial setting
the Moore's normalization functor $N$ is lax monoidal with respect to 
the Alexander-Whitney map $\AW\colon N(A)\otimes N(B)\ar N(A\otimes B)$ and 
is oplax monoidal with respect to the Eilenberg-Zilber map $\EZ\colon N(A\otimes B)\ar N(A)\otimes N(B)$. 
It is important that $\EZ$ and $\AW$ can be given by explicit formulas and, up to a homotopy,  are inverse to each other and symmetric.
This induces the corresponding lax/oplax structures on the inverse functor $K$.

In particular $$\mb{L}F\colon \big(D^{>0}_{perf}(\Ab),\oplus\big )\ar \big(D(\Mod(\kk)),\otimes\big)$$ 
is a strong symmetric monoidal functor. Using the isomorphism 
$\mb{L}F(P)\otimes \mb{L}F(P)\simeq \mb{L}F(P\oplus P)$ induced by $\EZ$ (the inverse of $\AW$),
we see that $\mb{L}F(P)$
is a Hopf algebra object in $D(\Mod(\kk))$.

Note that already in the simplest case $P=\mb{Z}[-n]$, the multiplication and the comultiplication on $\mb{L}F(\mb{Z}[-n])$ 
defined using the formula \eqref{eq:derfun} and well-known chain level-expressions for $\AW$ and $\EZ$,
\emph{a priori} rely on the whole Hopf algebra structure on $F(\mb{Z})$. Similarly, for two-term complex $P=[\mb{Z}\ar[p] \mb{Z}][-n]$
the very differential of the chain complex $\mb{L}F(P)$ involves the multiplication and comultiplication in $F(\mb{Z})$.

In this work we will show that if $P$ is sufficiently coconnected, then $\mb{L}F(P)$ depends only on the coalgebraic part of the Hopf algebra structure on $F(\mb{Z})$.
More precisely we have the following.
\begin{thrm*}[\ref{thrm:LF_LG_cs}]
	Assume $F,G$ are exponential functors
	provided with isomorphism 
	$F(\mb{Z})\simeq G(\mb{Z})$ of augmented coalgebras over $\kk$.
	Then there is an isomorphism of strong symmetric monoidal functors
	$$\mb{L}F\simeq \mb{L}G\colon \big(D^{<-1}_{perf}(\Mod(\kk))^{\vee},\oplus\big)\ar \big(D(\Mod(\kk)),\otimes\big).$$
\end{thrm*}

To state the other result, as an application of the technique developed in this work, 
recall that for any finitely generated abelian group $A\in \Ab^{fg}$ there is Dold-Puppe-Thom identification
of homology of the Eilenberg-MacLane space $K(A,n)$ \cite[Satz 4.16]{DP}:
\begin{equation}\label{eq:dpt}
	\DP\colon H_*(K(A,n);\mb{Z})\simeq \mb{L}_*\Sym(A[n]).
\end{equation}
The construction of $\DP$ uses an isomorphism 
of $\mb{L}\Sym(A[n])$ and the singular chains on the infinite symmetric power $\SP^{\infty}\MooreSp(A,n)$ of the Moore space $\MooreSp(A,n)$. 
It is well-known that in general there is no functor $A\leadsto \MooreSp(A,n)\in \Top$, \cite{Carlsson}.
In fact there is \emph{no} functorial isomorphism of $H_3(K(A,1);\mb{Z})$ and $\mb{L}_3 \Sym(A[1])$,
so the isomorphism \eqref{eq:dpt} is not functorial in general, see \cite[B.1]{Mikhailov} and \cite[3.1]{MikhailovSpl}.
It turns out that this phenomenon occurs only at the level of classical group homology:
\begin{thrm*}[\ref{cor:dtp}]
	The Dold-Puppe-Thom isomorphism $\DP$ is functorial for $n>1$.
\end{thrm*}

Finally, in \ref{sect:cochains_em} we will formulate the dual statement concerning the cohomology of Eilenberg-MacLane spaces $H^*(K(A,n);\mb{Z})$.
We will see that it is closely related to a surprising isomorphism of derived binomial and divided powers algebras \ref{sect:bin}:
$\mb{L}^*\Bin(\mc{RH}om(A[n],\mb{Z}))$ and $\mb{L}^*\Gamma(\mc{RH}om(A[n],\mb{Z}))$, respectively.
The last section \ref{sect:apps} includes two remarks about
the cohomology of iterated classifying stacks and the cohomology of symmetric products.
\bigskip

\textbf{Acknowledgements}: 
This work grew out of an attempt to understand an identification of the cohomology
$H^*(K(\mb{Z},3);\mb{Z})$ of the Eilenber-MacLane space given in
\cite[A]{KP}. 
I am grateful to Dmitry Kubrak for fruitful 
discussions and to Antoine To\'uze for helpful comments. 
Finally, I would like to thank Dmitry Kaledin for drawing my attention to Pirashvili's work, 
as well as Pierre Godfard, Alexey Gorinov, Bhanu Kiran, Grigory Kondyrev, Geoffroy Horel, Grigory~Papaynov, Georgii~Shuklin and Grigory Solomadin for relevant discussions.

\section{Preliminaries}
\subsection{Categories $\Ab'$ and $\Ab''$}
Consider the category $\Fin_*$ of finite based sets.
We denote its objects by $S_+/\pt=S\sqcup \{\pt\}$, 
where $\pt$ is the base point of $S_+$.
Define the functor $\mb{Z}\langle -\rangle\colon \Fin_*\ar \Ab$ 
by 
$$\mb{Z}\langle S_+/\pt\rangle:=\mb{Z}\langle S_+\rangle/\mb{Z}\langle \pt\rangle\simeq \mb{Z}^{\oplus S}.$$
\begin{dfn}
	The category $\Ab'$ is the image of $\mb{Z}\langle -\rangle\colon \Fin_*\lar \Ab^{\mathrm{free,fg}}$.
	Similarly let $\Ab''\subset \Ab^{\mathrm{free,fg}}$ be the image of $\Ab'$ under the contravariant 
	functor $V\ar V^\vee$.
\end{dfn}
Trivially $\Ab'$ is equivalent to $\Fin_*$, while $\Ab''$ is equivalent to $\Fin^{\op}_*$.
Denote by $I'\colon \Ab'\ar \Ab$ and by $I''\colon \Ab''\ar \Ab$ the corresponding faithful embeddings.

\begin{rmrk}

1. The objects of $\Ab'$ are free modules $\mb{Z}^{\oplus S}$ provided with a chosen unordered basis.
	A morphism $f_*\in \Hom_{\Ab'}(\mb{Z}^{S_1},\mb{Z}^{S_2})$ corresponds to a map 
	$f\colon S_1 \ar S_2\sqcup \{\pt\}$, such that $f_* e_s=e_{f(s)},s\in S_1$ if 
	$f(s)\neq \pt$ and zero otherwise.

\bigskip

2. Similarly a morphism $f^*\in \Hom_{\Ab''}(\mb{Z}^{S_1},\mb{Z}^{S_2})$ corresponds to a map
	$S_2\ar S_1\sqcup \{\pt\}$, such that $f^* e_{s}=\sum\limits_{s'\in f^{-1}(s)}e_{s'}$.

\end{rmrk}
Both categories $\Ab',\Ab''$ share morphisms $0\ar \mb{Z}$ and $\mb{Z}\ar 0$.
We see that $\mb{Z}^{\oplus 2}\ar[+] \mb{Z}$ belongs to $\Ab'$, while
$\mb{Z}\ar[\Delta] \mb{Z}^{\oplus 2}$ belongs to $\Ab''$.

We will consider $(\Ab',\oplus)$ and $(\Ab'',\oplus)$ as monoidal categories under the direct sum in $\Ab$.
More precisely, the monoidal structures correspond, respectively, to the coproduct in $\Fin_*$ and the product in $\Fin^{\op}_*$: 
${S_1}_+\vee {S_2}_+=(S_1\sqcup S_2)_+$.
Tautologically the monoidal structure $(\Ab',\oplus)$ is cocartesian and $(\Ab'',\oplus)$ is cartesian,
the functors $I'$ and $I''$ are strong symmetric monoidal.
We will be interested in monoidal functors of the form $(\Ab,\oplus)\lar (\Mod(\kk),\otimes)$.

Using the fact that the category of lattices $\Ab^{free,fg}$ is monoidal equivalent to the category with objects $[n],n\ge 0$ and the morphisms
$\Hom([n],[m])=\Mat_{m\times n}(\mb{Z})$, it is easy to observe the following identification (see \cite[5.2]{Pirashvili}, \cite[5.3]{Touze}).
\begin{prop}\label{prop:groups_as_functors}
	In each of the following cases, there is an equivalence of the category of strong symmetric monoidal functors
	\begin{enumerate}
		\item 
			$\{F\colon (\Ab^{\mathrm{free,fg}},\oplus)\lar (\Mod(\kk),\otimes)\}$ and $
		\{\text{Co/commutative Hopf algebras over }\kk\}$,
		\item 
			$\{F\colon (\Ab',\oplus)\lar (\Mod(\kk),\otimes)\}$ and $
			\{\text{Commutative algebras with augmentation over }\kk\}$,
		\item 
			$\{F\colon (\Ab'',\oplus)\lar (\Mod(\kk),\otimes)\}$ and $
			\{\text{Cocommutative coalgebras with augmentation over }\kk\}$.

	\end{enumerate}
	Each equivalence is given by the evaluation $F\leadsto F(\mb{Z})$.
\end{prop}

The algebro-geometrical way to describe the inverse equivalence in the case of Hopf algebras is as follows. 
Given a commutative affine group scheme $G$ over $\kk$, the inverse monoidal functor is given
by 
$\mc{O}_{G}^{\otimes -}\colon (\Ab^{free,fg},\oplus)\lar (\Mod(\kk),\otimes)$,
defined by 
\begin{equation}\label{eq:grp_sch}
	\mc{O}_G^{\otimes V}:=\mc{O}(G\otimes_{\ul{\mb{Z}}}V^\vee),
\end{equation}where $G$ and
the discrete group $V^\vee$ are
considered as $\ul{\mb{Z}}$-modules in schemes over $\kk$.
For example $\mc{O}_G^{\otimes \mb{Z}^n}\simeq \mc{O}(G^{\times n})$, so that $\mb{Z}^{2}\ar[+]\mb{Z}$
induces the multiplication in $\mc{O}_G$, while $\mb{Z}\ar[\Delta]\mb{Z}^2$ induces the comultiplication.

Let us mention that the application of this observation in the simplicial setting, provides an invariant which depends on the algebra  $F(\mb{Z})$
and a homotopy type $X$. This idea was developed \cite{PirashviliHHH} and goes back to Loday. For example, the case $X=S^1$ corresponds to Hochshild homology of 
the commutative algebra $F(\mb{Z})$ in the bimodule $\kk$.

\subsection{Derived functors}
Assume $F\colon \Ab\ar \Mod(\kk)$ is a functor, non necessary additive.
According to \cite[4.2.2.2]{Illusie} one can define its left derived functor $\mb{L}F\colon D^{-}(\Ab)\ar D(\Mod(\kk))$,
using cosimplicial-simplicial (mixed) version of Dold-Kan correspondence.
In this work it is enough to treat the simplical/connective and cosimplicial/coconnective ``regimes'' separately.
Below we will describe the functor $\mb{L}F\colon D^{>0}_{perf}(\Ab)\ar D(\Mod(\kk))$.
The simplical part $\mb{L}F\colon D^{\le 0}(\Ab)\ar D(\Mod(\kk))$ is similar (see \cite{Dold}).

Recall that there is an equivalence of categories
$$N\colon \cosimpl{\Ab}\substack{\lar\\\lal}\Ch^{\ge 0}(\Ab)\colon K,$$ given
by Moore's normalization functor $N(-)$ and the Kan functor $K(-)$,
provided with isomorphisms $K\circ N\simeq \id$ and 
$N\circ K\simeq \id$, see \cite{Weibel}.
For both categories one can define the notion of weak equivalence and of homotopy between morphisms, 
the functors $N,K$ preserve the weak equivalences and homotopies.
The localization of $\cosimpl{\Ab}$ by weak equivalences will be denoted by $\Ho(\cosimpl{\Ab})$
and that of cochain complexes by $D^{\ge 0}(\Ab)$.

Consider a functor $F\colon \Ab^{free,fg}\lar \Mod^{flat}(\kk)$.
A simple, but crucial observation is that
the induced functor $F\colon \cosimpl{\Ab^{free,fg}}\ar \cosimpl{\Mod^{flat}(\kk)}$ 
preserves homotopies \cite{Dold}.
Note that any $P\in D^{>0}_{perf}(\Ab)$ is quasi-isomorphic to a complex of free abelian groups of finite rank in 
$\Ch^{\ge 0}(\Ab^{free,fg})$.
By the above, the Kan functor $K$ induces an equivalence of $D^{>0}_{perf}(\Ab)$ and 
a full subcategory $\Ho^{>0,ft}(\cosimpl{\Ab})\subset \Ho(\cosimpl{\Ab})$.
Since any quasi-isomorphism between such complexes is a homotopy equivalence,
it follows
that the object $N(F(K(P)))\in D(\Mod(\kk))$ is well-defined up to unique isomorphism.
We will work with the following operational definition.
\begin{dfn}
	For $F$ as above, define
	its \emph{left derived functor} $\mb{L}F\colon D^{>0}_{perf}(\Ab)\ar D(\Mod(\kk))$
	by the formula $\mb{L}F(P):=N(F(K(P)))$, where $P\in \Ch^{\ge 0}(\Ab^{free,fg})$.
\end{dfn}
Thus $\mb{L}F$ is equal to the composition
\begin{equation}\label{eq:quillen_der_cs}
\begin{tikzcd}
	D^{>0}_{perf}(\Ab)\ar[r,"K"]& \Ho(\cosimpl{\Ab})\ar[r,"\mb{L}F_{cs}"]& \Ho(\cosimpl{\Mod(\kk)})\ar[r,"N"]& D(\Mod(\kk))
\end{tikzcd}
\end{equation}
where $\mb{L}F_{cs}\colon \Ho(\cosimpl{\Ab})\ar \Ho(\cosimpl{\Mod(\kk)})$ is the left derived functor or the right Kan extension of $F\colon \cosimpl{\Ab}\ar \Ho(\cosimpl{\Mod(\kk)})$ with 
respect to the model structure on $\cosimpl{\Ab}$
induced by the projective model structure on $\Ch^{\ge 0}(\Ab)$ via the Dold-Kan equivalence $K\colon \Ch^{\ge 0}(\Ab)\ar[\sim] \cosimpl{\Ab}$.
Concretely $\mb{L}F_{cs}$ is evaluated by applying $F$ to a cofibrant replacement.


Assume further that $F\colon (\Ab^{\mathrm{free,fg}},\oplus)\ar (\Mod(\kk),\otimes)$ is exponential, i.e. $F$ is provided with a 
strong symmetric monoidal structure.
The normalization functor $N$ is lax monoidal with respect to the Alexander-Whitney map
$$\AW\colon N(A)\otimes N(B)\ar N(A\otimes B),$$
and is oplax monoidal with respect to Eilenberg-Zilber map
$$\EZ\colon N(A\otimes B)\ar N(A)\otimes N(B),$$
for any $A,B\in \cosimpl{\Ab}$.
Recall that $\EZ$ is symmetric on the nose, 
while $\AW$ is only up to higher coherent homotopies.
In particular, since $K$ preserves the direct sums, the complex $N(F(K(P)))$ is provided with cocommutative comultiplication via $\EZ$,
while the multiplication given by $\AW$ is associative and $N(F(K(P)))$ is 
$E_\infty$-algebra.
Thus $\mb{L}F(P)\in D(\Mod(\kk))$ is a Hopf object, e.g.
the cohomology $\mb{L}^*F(P)$ is a graded Hopf algebra.

We conclude by noting that the connective part 
$\mb{L}F\colon D^{\le 0}_{perf}(\Ab)\ar D(\Mod(\kk))$ is treated similarly.
In particular we have the formula 
$\mb{L}F(P)=N(F(K(P)))$,
where $P\in \Ch^{\le 0}(\Ab^{free,fg})$.
The functor $\mb{L}F$ is equal to the composition
\begin{equation}\label{eq:quillen_der_s}
\begin{tikzcd}
	D^{\le 0}_{perf}(\Ab)\ar[r,"K"]& \Ho(\simpl{\Ab})\ar[r,"\mb{L}F_s"]& \Ho(\simpl{\Mod(\kk)})\ar[r,"N"]& D(\Mod(\kk))
\end{tikzcd}
\end{equation}
If $F$ is exponential, then the bialgebra structure on $\mb{L}F(P)\in D(\Mod(\kk))$ is induced by the maps $P\oplus P\ar[+] P$ and $P\ar[\Delta] P\oplus P$.

\section{Homotopy refinement of the Kan functor}
Given based simplicial set $(X,\pt)\in \simpl{\Fin}_*$ defines a free abelian group
$\mb{Z}\langle X\rangle\in \simpl{\Ab^{free,fg}}$. The
morphisms $\pt\ar X\ar \pt$ 
define a splitting 
$\mb{Z}\langle X\rangle\simeq \mb{Z}\langle X/\pt\rangle\oplus \mb{Z}\langle\pt \rangle$,
where the abelian group $\mb{Z}\langle X/\pt\rangle_n$ is spanned by 
an unordered basis given by elements of the form $x':=x-\pt$ in $\mb{Z}\langle X\rangle$, where
$x\in X_n\setminus \pt_n$.
Clearly $\mb{Z}\langle X/pt\rangle\in \simpl{\Ab'}$,
thus we obtain a functor $\mb{Z}\langle-/\pt \rangle\colon \simpl{\Fin}_*\ar \simpl{\Ab'}$.
\subsection{Moore spaces}
\begin{dfn}\label{dfn:moore_sp}
        Given abelian group $A\in \Ab^{\mathrm{fg}}$ and $n>1$,
	the Moore space $\MooreSp(A,n)\in \Top_*$
                is a 1-connected space such that
        $\tilde{H}_i(\MooreSp(A,n);\mathbb{Z})=0,i\neq n$, provided
        with an isomorphism $H_n(\MooreSp(A,n);\mathbb{Z})\simeq A$.
\end{dfn}
We can represent
$\MooreSp(A,n)$ as a CW-complex with cells in dimensions $n,n+1$ as follows.
Take a two-term resolution $[\mb{Z}^R\ar[f] \mb{Z}^G]\ar[\sim] A$ and set 
\begin{equation}\label{moore_cone}
M(A,n):=\cof(\vee_R S^n\ar[F] \vee_G S^n)\simeq \vee_R D^{n+1}\sqcup_{F} \vee_G S^n,
\end{equation}
where the map $F$ is such that
$H_n(F;\mb{Z})=f$.
In fact $\MooreSp(A,n)$ can be modeled by 
a finite simplicial complex.
Moreover, if $F\colon (\MooreSp(A,n),\pt)\to (\MooreSp(A',m),\pt)$ is a map of 
the topological spaces corresponding to simplicial complexes, then
the simplicial approximation theorem asserts 
that $F$ is homotopic to a simplicial map 
$F^\varepsilon:(\MooreSp(A,n),\pt)\to (\MooreSp(A',m),\pt)$
after passing
to sufficiently fine barycentric subdivision.
Thus any topological statement about maps
between Moore space has a simplicial couterpart. Abusing notation we will
denote any such finite simplicial model of the Moore space by 
$\MooreSp(A,n)\in \simpl{\Fin}_*$.

Note that the homotopy type of $\MooreSp(A,n)$ is unique for $n>1$,
but there is \emph{no} functor $\MooreSp(A,n)\colon \Ab^{fg}\ar \Top$ \cite{Carlsson}.
The situation is better if one passes to the homotopy category $\Ho(\Top)$ and invert $2$, see \cite[cor 6.6]{Neisendorfer},\cite{Neeman} and \ref{rmrk:moore_functor}.
However, there are sufficiently many maps between Moore spaces.
\begin{prop}\label{prop:maps_of_moore}
        For $n>1$ and $A,A'\in \Ab^{\mathrm{fg}}$ the natural maps
        \begin{equation}\label{eq:moore_hom}
                \Hom_{\Ho(\Top_*)}(\MooreSp(A,n),\MooreSp(A',n))
                \overset{C^{\mathrm{sing}}_*(-)}{\lar} 
                \Hom_{D(\Ab)}(A[n],A'[n])=\Hom_{\Ab}(A,A')
        \end{equation}
        \begin{equation}\label{eq:moore_ext}
                \Hom_{\Ho(\Top_*)}(\MooreSp(A,n),\MooreSp(A',n+1))
		\overset{C^{\mathrm{sing}}_*(-)}{\lar}
		\Hom_{D(\Ab)}(A[n],A[n+1])=\Ext^1_{\Ab}(A,A')
        \end{equation}
        are surjective.
\end{prop}
\begin{proof}
	Let $\MooreSp(A,n)$ be the cellular space defined by \eqref{moore_cone}.
	Let 
	$$C^A_*:=C^{\mathrm{cw}}_*(M(A,n))\simeq [\mb{Z}^R\ar \mb{Z}^G][n],$$ and similarly for $A'$.

1. Given $u\in \Hom_{\Ab}(A,A')$, it lifts to a chain morphism of complexes
	$\tilde{u}\colon C^A_*\ar C^{A'}_*$.
        For $n>1$ there is a diagram of spaces:
        \begin{displaymath}
        \begin{tikzcd}
        \bigvee_R S^n\arrow[r]\arrow[d,"U_n"]& 
		\bigvee_G S^n\arrow[r]\arrow[d,"U_{n+1}"]& 
                        \MooreSp(A,n)\arrow[d,dashed,"U"]\\
        \bigvee_{R'} S^n\arrow[r]&
                \bigvee_{G'} S^n\arrow[r]&
                        \MooreSp(A',n)
        \end{tikzcd}
        \end{displaymath}
	where $U_k,k=n,n+1$ are maps determined, up to homotopy, by the condition $\tilde{u}_k=H_n(U_k)$.
        The left square is homotopy commutative,
        thus, by basic properties of cofiber sequences,
	the dashed arrow $U$ exists and we have $H_n(U)=u$.
        This proves surjectivity in \eqref{eq:moore_hom}.
	
2. 
Denote by $[-,-]$ the set of based maps up to based homotopy.
Consider 1-connected topological space $Y$.
Applying $\Maps_*(-,Y)$ to the cofiber sequence 
	$$\vee_R S^n\ar \vee_G S^n\ar \MooreSp(A,n)\ar \vee_R \Sigma S^n\ar \vee_G \Sigma S^{n}$$ gives an exact sequence
	of abelian groups:
	$$[\vee_G S^{n+1},Y]\ar[] [\vee_R S^{n+1},Y]\ar[] [\MooreSp(A,n),Y]\ar[] [\vee_G S^{n},Y]\ar[] [\vee_R S^{n},Y],$$
which is equivalent to the short exact sequence:
	$$0\ar \Ext^1(A,\pi_{n+1}Y)\ar[] [\MooreSp(A,n),Y]\ar \Hom(A,\pi_n Y)\ar 0.$$
For $Y=\MooreSp(A',n+1)$ we obtain a natural isomorphism $\Ext^1(A,A')\simeq [\MooreSp(A,n),\MooreSp(A',n+1)]$.
\end{proof}
\begin{rmrk}\label{rmrk:moore_functor}
	Notice that the map $U\colon \MooreSp(A,n)\ar \MooreSp(A',n)$ in part 1 of the proof is constructed by means of the homotopy making the left square strictly commutative. Thus 
	the indeterminancy in the choice of $U$ is controlled by the group of homotopy clases 
	$[\vee_R S^{n+1},\vee_{G'}S^{n}]$. In particular, \eqref{eq:moore_hom} becomes an isomorphism
	after inverting $2$ for $n>2$. See also \cite[1.3.10]{Baues}.
\end{rmrk}
\begin{rmrk}
        For $n=1$ the surjectivity of \eqref{eq:moore_hom} fails.
        For a counterexample consider non-simply connected Moore space
        $\mathbb{R}P^2$ for the group $A=\mathbb{Z}/2$.
        Suppose 
        $\Delta\colon \mathbb{R}P^2\to \mathbb{R}P^2\vee \mathbb{R}P^2$
        is a map such that
	$H_1(\Delta)\in \Hom_{\Ab}(A,A\oplus A)$ is the diagonal.
        Denote by $\pi_i\colon \mathbb{R}P^2\vee \mathbb{R}P^2\to \mathbb{R}P^2$ for $i=1,2$,
        the pair of collapsing maps.
        Let $\gamma\in H^1(\mathbb{R}P^2,\mathbb{Z}/2)$ be the generator.
        Then $\pi_i\circ \Delta$ induces the identity on
        $H_1(\mathbb{R}P^2,\mathbb{Z}/2)$ and
        hence on its dual $H^1(\mathbb{R}P^2,\mathbb{Z}/2)$.
        Thus
        $(\pi_1\circ \Delta)^*\gamma\cup (\pi_2\circ \Delta)^*\gamma$ 
	is equal to $\gamma^2\neq 0\in H^2(\mathbb{R}P^2,\mathbb{Z}/2)$.
        On the other hand it is equal to
        $\Delta^*(\pi_1^*\gamma\cup \pi_2^*\gamma)=0$,
	thus no such $\Delta$ exists.
\end{rmrk}

Fix a (non functorial) Moore space construction $\MooreSp(-)$.
For $P_*\in D^{<-1}_{\mathrm{perf}}(\Ab)$ we generalize it by setting
\begin{equation}\label{general_moore}
\MooreSp(P_*):=\bigvee_i
		\MooreSp(H_i(P_*),i)\in \simpl{\Fin}_*.
\end{equation}
One can always assume that $\MooreSp(P_*)$ has no non-trivial simplices
in degrees $\leq 1$, thus
by Whitehead's theorem any two such constructions $\MooreSp(P_*)$
and $\MooreSp'(P_*)$ are homotopy equivalent and are suspensions spaces.
In particular $\MooreSp(P_*)$ is a comonoid in the homotopy category of simplicial sets.
For $X\in \simpl{\Fin}_*$ with the base point $\pt\in X$, call
$$\tilde{C}_*(X)=N(\mb{Z}\langle X/\pt\rangle),$$
the complex of reduced chains of $X$.

Note that there is a non canonical isomorphism 
\begin{equation}\label{der_splt}
	P\simeq \oplus_i H_i(P)[i]\in D(\Ab)
\end{equation} which is identity on $H_i(-)$.
This provides an isomorphism $\tilde{C}_*(\MooreSp(P))\simeq P\in D^{<-1}_{perf}(\Ab)$.

\begin{cor}\label{cor:moore}
	Given $P,Q\in D^{<-1}_{perf}(\Ab)$ fix the Moore spaces $\MooreSp(P),\MooreSp(Q)$.
	The natural map
	$$\Maps_{\simpl{\Fin}_*}(\MooreSp(P),\MooreSp(Q))\lar \Hom_{D^{<-1}_{perf}(\Ab)}(\tilde{C}_*(\MooreSp(P)),\tilde{C}_*(\MooreSp(Q))),$$
	is surjective.
\end{cor}
\begin{proof}
	Using simplicial approximation it is enough to prove the surjectivity of the map
	\begin{equation}\label{eq:claim}
		[\MooreSp(P),\MooreSp(Q)]\ar {[\tilde{C}^{sing}_*(\MooreSp(P)),\tilde{C}^{sing}_*(\MooreSp(Q))]},
	\end{equation}
	where the brackets denote the morphisms in the categories $\Ho(\Top_*)$ and $D(\Ab)$, respectively.
	Using the fact that $\MooreSp(P)$ is a comonoid, the map above is a group homomorphism.
	The identification $P\simeq \oplus_i H_i(P)[i]$ chosen in \eqref{general_moore}, and similarly for $Q$,
	establish the following commutative diagram
\begin{equation}
\begin{tikzcd}
	{[\MooreSp(P),\MooreSp(Q)]}\ar[r] & {[P,Q]\simeq [\tilde{C}_*\big(\MooreSp(P)\big),\tilde{C}_*\big(\MooreSp(Q)\big)]}\\
	{[\MooreSp(H_i(P),i), \MooreSp(H_j(Q),j)]}\ar[u]\ar[r] & {[H_i(P)[i],H_j(Q)[j]]\simeq [\tilde{C}_*\big(\MooreSp(H_i(P),i)\big),\tilde{C}_*\big(\MooreSp(H_j(Q),j)\big)]}\ar[u]
\end{tikzcd}
\end{equation}
	By exactness of the sequence
	$$0\ar \bigoplus_i \Ext^1(H_i(P),H_{i+1}(Q))\lar \Hom_{D(\Ab)}(P,Q)\lar \bigoplus_i \Hom(H_i(P),H_i(Q))\ar 0,$$
	the surjectivity of \eqref{eq:claim} follows from surjectivity 
	of the map $[\MooreSp(A,i),\MooreSp(A',j)]\ar {[A[i],A'[j]]}$ for all $A,A'\in \Ab^{fg}$ and $j=i,i+1$.
	This case is covered by
	\ref{prop:maps_of_moore}.
\end{proof}

\subsection{Functors $K'$ and $K''$}
Recall that the Dold-Kan functor $K\colon \Ch^{\le 0}(\Ab)\ar \simpl{\Ab}$
preserves weak-equivalences and induces the equivalence 
$K\colon D^{\le 0}(\Ab)\ar \Ho(\simpl{\Ab})$,
where $\Ho(\simpl{\Ab})$ denotes the localization of $\simpl{\Ab}$ with respect
to homotopy equivalences of geometric realizations of the simplicial sets.
By definition the embedding functor $I'\colon \Ab'\ar \Ab$ is equivalent to  
$\mb{Z}\langle -/\pt\rangle\colon \Fin_*\simeq \Ab'\ar \Ab$.
The induced functor $\simpl{\Ab'}\ar \Ho(\simpl{\Ab})$ admits a factorization
\begin{equation}\label{fact'}
	\simpl{\Ab'}\ar[p'] \Ho(\simpl{\Ab'})\ar[i'] \Ho(\simpl{\Ab}),
\end{equation}
where $i'$ is faithful, and $p$ is full and essentially surjective.
Concretely, we define $\Ho(\simpl{\Ab'})$ as the quotient category $\simpl{\Ab'}/\sim$, where
the equivalence relation $\sim$ identifies a pair of arrows $u,v\in \Hom_{\simpl{\Ab'}}(X,Y)$ iff $I'(u)=I'(v)\in \Hom_{\Ho(\simpl{\Ab})}(X,Y)$.
Notice that though $\Ab'\simeq \Fin_*$, the natural functor 
from the homotopy category of spaces $\Ho(\simpl{\Fin}_*)$ to $\Ho(\simpl{\Ab'})$ is not faithful. 

\begin{theorem}\label{thrm:good_s_dold_kan}
        There is a strong symmetric monoidal functor
	$$K'\colon \big(D^{<-1}_{\mathrm{perf}}(\Ab),\oplus\big)\ar \big(\Ho(\simpl{\Ab'},\oplus)\big),$$
        defined up to a natural isomorphism, such that
        $i'\circ K'$ is
	naturally isomorphic, as a monoidal functor, to 
	the Kan functor $K\colon \big(D^{<-1}_{perf}(\Ab),\oplus\big)\ar \big(\Ho(\simpl{\Ab},\oplus)\big)$.
\end{theorem}
\begin{proof}
	Denote by $\Ho^{<-1,ft}(\simpl{\Ab})$ the full subcategory of $\Ho(\simpl{\Ab})$
	spanned by 1-connected simplicial groups of finite type, i.e. such that the direct sum of all homotopy groups $\oplus_i \pi_i$ is finitely generated.
	Consider the following commutative diagram of categories
\begin{equation}
\begin{tikzcd}
	\Ho^{<-1,ft}(\simpl{\Ab})\ar[r,"K"]& D^{<-1}_{perf}(\Ab)\\
	\Ho^{<-1,ft}(\simpl{\Ab'})\ar[u,"i'"]\ar[ru]
\end{tikzcd}
\end{equation}
	Here $i'$ is faithful, $K$ is Dold-Kan equivalence of homotopy categories, while the diagonal arrow
	equal to $K\circ i'$
	is full and essentially surjective by \ref{cor:moore}. 
	Thus $\Ho^{<-1,ft}(\simpl{\Ab'})\ar D^{<-1}_{perf}(\Ab)$ is an equivalence of categories.
	Let $K'$ be an inverse of the diagonal arrow followed by the embedding $\Ho^{<-1,ft}(\simpl{\Ab'})\ar \Ho(\simpl{\Ab'})$.
	Clearly $K'$ preserve finite coproducts. 
	Since the monoidal structures on the underlying categories are cocartesian, the functor $K'$ 
	is strong symmetric monoidal with respect to $\oplus$, we have the isomorphism $i'\circ K'\simeq K$.
\end{proof}
Let $D^{<-1}_{perf}(\Ab)^\vee$ denotes the image of
$D^{<-1}_{perf}(\Ab)$ under the functor $\mc{RH}om_{D(\Ab)}(-,\mb{Z})$.
The functor $\Hom_{\Ab}(-,\mathbb{Z})$
induces an equivalence $\Ab'^{\op}\ar[\sim] \Ab''$
in $\Ab$. Similarly we have an equivalence $D^{<-1}_{perf}(\Ab)^{\op}\ar[\sim] D^{<-1}_{perf}(\Ab)^{\vee}$.
By the definition the monoidal structure $(\Ab',\oplus)$ (resp. $(\Ab'',\oplus)$) is cocartesian (resp. cartesian):
in both cases it is the direct sum in $\Ab$.
Similarly to \eqref{fact'}, define $\Ho(\cosimpl{\Ab}'')$ by the factorization:
$$\cosimpl{\Ab''}\ar[p''] \Ho(\cosimpl{\Ab}'')\ar[i''] \Ho(\Ab),$$
into faithful $i''$, and full and essentially surjective $p''$.

For  $P\in D^{<-1}_{perf}(\Ab)^{\vee}$, let $P^{\vee}=\mc{RH}om_{D(\Ab)}(P,\mb{Z})\in D^{<-1}_{perf}(\Ab)$.
As a direct corollary we obtain:
\begin{cor}\label{thrm:good_cs_dold_kan}
        The functor
	$$K''\colon (D^{<-1}_{\mathrm{perf}}(\Ab)^\vee,\oplus)\to (\Ho(\cosimpl{\Ab''}),\oplus),$$
	defined by the formula $K''(P)=K'(P^{\vee})^{\vee}$
        is strong symmetric monoidal;
        $i''\circ K''$ is isomorphic,
	as a monoidal functor,
	to the Kan functor $\big(D^{<-1}_{perf}(\Ab)^{\vee},\oplus\big)\ar \big(\Ho(\cosimpl{\Ab},\oplus) \big) $.
\end{cor}

\subsection{Equivalence of left derived functors}
Assume $F\colon \Ab\ar \Mod^{flat}(\kk)$ is a functor.

\bigskip

1. 
Denote by $F'\colon \Ab'\ar \Ab$ the restriction of $F$ along the embedding $I'\colon \Ab'\ar \Ab$.
By \eqref{eq:quillen_der_s} and Theorem \ref{thrm:good_s_dold_kan}, the functor $\mb{L}F\colon D^{<-1}_{perf}(\Ab)\ar D(\Mod(\kk))$ is equal to the composition
\begin{equation}
\begin{tikzcd}
	D^{<-1}_{perf}(\Ab)\ar[r,"K'"]\ar[rr,bend left=20, "K"]& \Ho(\simpl{\Ab'})\ar[r,"i'"]& \Ho(\simpl{\Ab})\ar[r,"\mb{L}F_s"]& \Ho(\simpl{\Mod(\kk)})\ar[r,"N"]& D(\Mod(\kk))
\end{tikzcd}
\end{equation}
where $\mb{L}F_s\colon \Ho(\simpl{\Ab})\ar \Ho(\simpl{\Mod(\kk)})$ is evaluated by applying $F$ to a cofibrant replacement in the model category $\simpl{\Ab}$.

\begin{lemma}\label{lemma:key}
	The composition
\begin{equation*}
\begin{tikzcd}
	\simpl{\Ab'}\ar[r,"p'"]&\Ho(\simpl{\Ab'})\ar[r,"i'"]& \Ho(\simpl{\Ab})\ar[r,"\mb{L}F_s"]& \Ho(\simpl{\Mod(\kk)})
\end{tikzcd}
\end{equation*}
	is naturally isomorphic to $F'\colon \simpl{\Ab'}\ar \Ho(\simpl{\Mod(\kk)})$.
Moreover, if $F$ is a strong symetric monoidal functor: $(\Ab^{free,fg},\oplus)\ar (\Mod(\kk),\otimes)$,
	then the isomorphism $F'\simeq \mb{L}F_s\circ i'\circ p'$ of functors
	$(\simpl{\Ab'},\oplus)\ar (\Ho(\simpl{\Mod(\kk)}),\otimes)$ is monoidal.
\end{lemma}
\begin{proof}
	Consider the following diagram
\begin{equation}
\begin{tikzcd}
\Ho(\simpl{\Ab'})\ar[r,"i'"] & 
		\Ho(\simpl{\Ab})\ar[r,"\mb{L}F_s"] & 
			\Ho(\simpl{\Mod(\kk)})\\
	\simpl{\Ab'}\ar[u,"p'"]\ar[r,"I'"]\ar[rru,bend right=20,"F'" {near end}]\ar[rd,dashed,swap,"N\circ I'"]& 
		\simpl{\Ab}\ar[u] &\\
			    &
		\Ch^{\le 0}(\Ab^{free,fg})\ar[u,"K"]\ar[uur,swap,"F\circ K"]&
\end{tikzcd}
\end{equation}
The square obviously commutes. Up to an isomorphism of functors, 
the right triangle commutes because the image of $K$ consists of cofibrant objects in $\simpl{\Ab}$ and hence
$F\circ K$ is naturally isomorphic to $\mb{L}F_s\circ K$.
	Since the normalized chains of a free simplicial abelian group is a complex of free abelian groups, $N\circ I'$ lands in 
	$\Ch^{\le 0}(\Ab^{free,fg})\subset \Ch^{\le 0}(\Ab)$. Hence the straight dashed arrow exists and makes the bottom triangle commute.
Then the bended arrow is naturally isomorphic to $F\circ K\circ N\circ I'\simeq F\circ I'=F'$.
Thus $\mb{L}F_s\circ i'\circ p'$ is naturally isomorphic to $F'$.
If $F$ is monoidal, then the natural isomorphisms of functors discussed above are monoidal as well. This concludes the proof.
\end{proof}
It follows that $\mb{L}F_s\circ i'$ evaluated on $E\in \simpl{\Ab'}$ is naturally isomorphic to $F'(E)\in \Ho(\simpl{\Mod(\kk)})$, because $E$ is cofibrant in $\simpl{\Ab}$.
Hence the functor $N\circ \mb{L}F_s\circ i'\circ K'\simeq \mb{L}F$ is determined by the restriction $F'\colon \Ab'\ar \Mod(\kk)$.
As a direct corollary we obtain:
\begin{thrm}\label{thrm:LF_LG}
	Suppose the pair of functors $F,G\colon \Ab^{\mathrm{fg,free}}\ar \Mod^{flat}(\kk)$
	is provided with a morphism $\phi'\colon F'\ar G'$ of functors $\Ab'\ar \Mod^{flat}(\kk)$.
                    Then $\phi'$ induces a natural morphism
                of left derived functors
                $$\phi\colon \mb{L}F\ar \mb{L}G,$$
		from $D^{<-1}_{\mathrm{perf}}(\Ab)$ to $D(\Mod(\kk))$.
                Moreover, if the functors
                $F,G\colon (\Ab^{\mathrm{fg,free}},\oplus)\ar 
			(\Mod^{\mathrm{flat}}(\kk),\otimes)$
		are provided with a strong symmetric monoidal structure and
		$\phi'\colon F'\ar G'$ is monoidal, then 
		$\phi\colon \mb{L}F\ar \mb{L}G$ is monoidal as well.
\end{thrm}
\begin{proof}
By key lemma \ref{lemma:key},
	$\phi'\colon F'\ar G'$ induces a morphism $\mb{L}F_s\circ i'\ar \mb{F}G_s\circ i'$. Define $\phi\colon \mb{L}F\ar \mb{L}G$ by
$$N\circ \mb{L}F_s\circ i'\circ K'\simeq N\circ\mb{L}F_s\circ K=\mb{L}F\lar N\circ \mb{L}G_s\circ i'\circ K'\simeq N\circ\mb{L}G_s\circ L=\mb{L}G.$$

In the case $F$ is exponential, i.e. $F$ is provided with strong symmetric monoidal structure,
the isomorphism $\mb{L}F_s\circ i'\circ p'\simeq F'$ is monoidal. Hence any symmetric monoidal morphism $\phi'\colon F'\ar G'$
induces monoidal $\phi$.
\end{proof}
In other words, if $F,G$ are exponential functors such that the augmented $\kk$-algebras $F'(\mb{Z})$ and $G'(\mb{Z})$ are
isomorphic, then there is an isomorphism of exponential functors $\mb{L}F\simeq \mb{L}G$ with the domain of definition as above.

\bigskip

2. Below we will formulate the dualization of \ref{lemma:key} and \ref{thrm:LF_LG} to $1$-coconnective case,
the proofs are similar to the above.
Let $F''\colon \Ab''\ar \Ab$ denote the restriction of $F$ along the embedding $I''\colon \Ab''\ar \Ab$.
\begin{lemma}\label{lemma:key_cs}
	The composition
\begin{equation*}
\begin{tikzcd}
	\cosimpl{\Ab''}\ar[r,"p''"]&\Ho(\cosimpl{\Ab''})\ar[r,"i''"]& \Ho(\cosimpl{\Ab})\ar[r,"\mb{L}F_{cs}"]& \Ho(\cosimpl{\Mod(\kk)})
\end{tikzcd}
\end{equation*}
	is naturally isomorphic to $F''\colon \cosimpl{\Ab''}\ar \Ho(\cosimpl{\Mod(\kk)})$.
Moreover, if $F$ is a strong symetric monoidal functor: $(\Ab^{free,fg},\oplus)\ar (\Mod(\kk),\otimes)$,
	then the isomorphism $F''\simeq \mb{L}F_{cs}\circ i''\circ p''$ of functors
	$(\cosimpl{\Ab}'',\oplus)\ar (\Ho(\simpl{\Mod(\kk)}),\otimes)$ is monoidal.
\end{lemma}
\begin{thrm}\label{thrm:LF_LG_cs}
	Suppose the pair of functors $F,G\colon \Ab^{\mathrm{fg,free}}\ar \Mod^{flat}(\kk)$
	is provided with a morphism $\phi''\colon F''\ar G''$ of functors $\Ab''\ar \Mod^{flat}(\kk)$.
                    Then $\phi''$ induces a natural morphism
                of left derived functors
                $$\phi\colon \mb{L}F\ar \mb{L}G,$$
		from $D^{<-1}_{\mathrm{perf}}(\Ab)^{\vee}$ to $D(\Mod(\kk))$.
		If $\phi''$ is isomorphism, then $\phi$ is isomorphism as well.
                Moreover, if the functors
                $F,G\colon (\Ab^{\mathrm{fg,free}},\oplus)\ar 
			(\Mod^{\mathrm{flat}}(\kk),\otimes)$
		are provided with a strong symmetric monoidal structure and
		$\phi''$ is monoidal, then 
		$\phi$ is monoidal as well.
\end{thrm}

As a practical application we have.
\begin{cor}\label{cor:LF_LG_algebras}
	Given exponential functors $F,G$ provided with an isomorphism of augmented coalgebras $F''(\mb{Z})\simeq G''(\mb{Z})$,
        for any $P\in D^{-1}_{\mathrm{perf}}(\Ab)^\vee$,
	there is
        an isomorphism of Hopf algebras 
	$\phi(P)\colon \mb{L}F(P)\simeq \mb{L}G(P)$ in $D(\Mod(\kk))$.
	The isomorphism 
        is natural in $P$.
\end{cor}

The isomorphism $\phi(P)$ can be described as follows.
Given $P\in D^{<-1}_{perf}(\Ab)^{\vee}$, $K(P^{\vee})\in \simpl{\Ab}$ is isomorphic to 
$\mb{Z}\langle X/\pt\rangle$
for some generalized Moore space $X\simeq \MooreSp(P)$.
By \ref{cor:moore},
the morphisms $\mb{Z}\langle X/\pt\rangle\oplus \mb{Z}\langle X/\pt\rangle\ar[+] \mb{Z}\langle X/\pt\rangle$ 
and 
$\mb{Z}\langle X/\pt\rangle\ar[\Delta] \mb{Z}\langle X/\pt\rangle\oplus \mb{Z}\langle X/\pt\rangle$
in $\Ho(\simpl{\Ab})$, 
are induced by some based maps $u\colon X/\pt\vee X/\pt\ar X/\pt$ and $v\colon X/\pt\ar X/\pt\vee X/\pt$.
For example one can take $u$ equal to the projection, the existence of $v$ follows from the fact that $X$ is a suspension space.
So $U=\mb{Z}\langle u/\pt\rangle^{\vee}$ and $V=\mb{Z}\langle v/\pt\rangle^{\vee}$ are morphisms in $\cosimpl{\Ab''}$.
Consider the diagram
\begin{equation}
\begin{tikzcd}
	F(\mb{Z}\langle X/\pt\rangle^{\vee})^{\otimes 2}\ar[r,"\phi''\otimes \phi''"]\ar[d]
		& G(\mb{Z}\langle X/\pt\rangle^{\vee})^{\otimes 2}\ar[d]\\
	F(\mb{Z}\langle X/\pt\vee X/\pt\rangle^{\vee})\ar[r,"\phi''"]\ar[d,"F''(V)"]
		& G(\mb{Z}\langle X/\pt\vee X/\pt\rangle^{\vee})\ar[d,"G''(V)"]\\
	F(\mb{Z}\langle X/\pt\rangle^{\vee})\ar[r,"\phi''"]
		& G(\mb{Z}\langle X/\pt\rangle^{\vee})
\end{tikzcd}
\end{equation}
Here the top square is constructed using monoidal structure on $F'$,$G'$ and the identification
$\mb{Z}\langle X/\pt\vee X/\pt\rangle\simeq \mb{Z}\langle X/\pt\rangle^{\oplus 2}$. The square commutes since $\phi''$ is monoidal.
The bottom square commutes because $\phi''$ is a tranformation of functors.

Thus
$\phi''\colon N(F(\mb{Z}\langle X/\pt\rangle^{\vee}))\simeq N(G(\mb{Z}\langle X/\pt\rangle^{\vee}))$
is an isomorphism of algebras in $D(\Mod(\kk))$.
Similar considerations with $U$ imply that there is a well-defined isomorphism 
of bialgebras $\phi(P)\colon \mb{L}F(P)\simeq\mb{L}G(P)$ in $D(\Mod(\kk))$.

\begin{rmrk}
	The assumption on cohomology of $P$ is important. Denote by $K(A,n)$ the Eilenberg-MacLane space
	for an abelian group $A$. In the following section 
	we will introduce monoidal functors $\Bin$ 
	and $\Gamma$ such that $\Bin''\simeq \Gamma''$. There are isomorphisms
	$$\mb{L}^*\Bin(\mb{Z}/2[-2])\os{L}{\otimes} \mb{Z}/2\simeq H^*(K(\mb{Z}/2,1);\mb{Z}/2),$$
	and 
	$$\mb{L}^*\Gamma(\mb{Z}/2[-2])\os{L}{\otimes} \mb{Z}/2\simeq 
	H^*(K(\mb{Z},1);\mb{Z}/2)\otimes H^*(K(\mb{Z},2);\mb{Z}/2).$$
	This shows that the algebras $\mb{L}^*\Bin(\mb{Z}/2[-2])$ and $\mb{L}^*\Gamma(\mb{Z}/2[-2])$ are not
	isomorphic.
\end{rmrk}
\begin{rmrk}
	One can construct an isomorphism $\mb{L}F(\mb{Z}[-n])\simeq \mb{L}G(\mb{Z}[-n])$ directly.
	Namely $K(\mb{Z}[-n])$ is equal to the cochains of $\Delta^n/\partial \Delta^n$, hence 
	$K(\mb{Z}[-n])\in \cosimpl{\Ab''}$ and one can check that the complex $N(F(K(\mb{Z}[-n])))$ involves only coalgebra structure
	of $F(\mb{Z})$. 
	We get a tautological additive isomorphism of 
	$\mb{L}F(\mb{Z}[-n])$ and $\mb{L}G(\mb{Z}[-n])$. 
	Then a straighforward computation shows that the multiplication in $N(F(K(\mb{Z}[-n])))$, given by the Alexander-Whitney 
	formula for $\AW$, does not involve the multiplication in $F(\mb{Z})$ at all!
	Remarkably there is no evident comparsion morphism on the chain level 
	from $N(F(K(P)))$ to $N(G(K(P)))$
	already in the case $P=\mb{Z}/2[-n]$.
\end{rmrk}

\section{Functoriality of Dold-Puppe-Thom isomorphism}
Recall that $K\colon \Ch^{\le 0}(\Ab)\ar \simpl{\Ab}$ is the Kan functor.
Let $|-|\colon \simpl{\Ab}\ar \Top_*$ denotes the geometrical realization.
\begin{dfn}
	Given $P\in \Ch^{<0}(\Ab)$, the generalized Eilenberg-MacLane space is
	$$\EM(P)=|K(P)|\in \Top_*.$$
\end{dfn}
Clearly $\EM(P)$ is a functor in $P$. Since $K$ and $|-|$ preserve weak equivalences, 
the functor $\EM$ 
descends to the localizations: $\EM\colon D^{<0}(\Ab)\ar \Ho(\Top_*)$.
Similarly we can define $\EM\colon D^{<0}_{perf}(\Ab)\ar \Ho(\simpl{\Fin}_*)$.
For $A\in \Ab$, $\EM(A[n])$ is naturally homotopy equivalent to the usual Eilenberg-MacLane space $K(A,n)$.

The $\mb{Z}$-span is a strong symmetric monoidal functor
$\mb{Z}\langle -\rangle\colon (\Ab,\oplus)\ar (\Mod(\mb{Z})^{flat},\otimes),$
induces the corresponding left derived functor
$$\mb{L}\mb{Z}\langle -\rangle\colon (D^{<0}_{perf}(\Ab),\oplus)\ar (D(\Mod(\mb{Z})),\otimes).$$
Recall the following.
\begin{prop}\label{prop:chains_em}
	There is natural monoidal isomorphism
	$$C^{sing}_*(\EM(-))\simeq \mb{L}\mb{Z}\langle -\rangle,$$
	of strong symmetric monoidal functors from 
	$(D^{<0}_{perf}(\Ab),\oplus)$ to $D((\Mod(\mb{Z})),\otimes)$.
\end{prop}
\begin{proof}
	The cellular and singular chains on $\EM(P)$, with respect to the natural cw-structure, are quasi-isomorphic
	to $N(\mb{Z}\langle |K(P)|\rangle)$ with respect to the symmetric monoidal structures. Note that $\mb{Z}\langle |-|\rangle$
	passes through the forgetful functor $oblv\colon \simpl{\Ab}\ar \simpl{\Set}_*$.
	Since $oblv$ lands in Kan complexes and any weak equivalence of Kan complexes is a homotopy equivalence, it follows that 
	$\mb{Z}\langle|-|\rangle$ preserves weak equivalences of simplicial groups, hence
	$\mb{L}\mb{Z}\langle P\rangle$ is naturally isomorphic to $N(\mb{Z}\langle K(P)\rangle)$, i.e. 
	no cofibrant replacement is needed.
\end{proof}
In particular, the homology of Eilenberg-MacLane spaces $H_i(K(A,n))$ is naturally isomorphic to $\mb{L}_i\mb{Z}\langle A[n]\rangle$.

For a based set $(X,\pt)$ consider $n$-th symmetric power $\SP^n X=X^{\times n}/\Sigma_n$.
The base point induces the inclusions $\SP^n X\ari \SP^{n+1} X$, thus $\SP^n X/\SP^{n-1} X$ is the quotient 
$(X,\pt)^{\wedge n}/\Sigma_n$ of $n$-th smash product by the action of the permutation group.
Let 
$\SP^{\infty}(X,\pt)=\colim_{n}\SP^{n}X.$
We obtain strong symmetric monoidal functor 
$$\SP^{\infty}\colon (\Fin_*,\vee)\ar (\Set_*,\times).$$
Thus $\SP^{\infty}(X,\pt)$ is a free commutative monoid generated by $(X,\pt)$ with the unit $\pt$.
Similarly, we have a strong monoidal functor
$$|\mb{Z}\langle-/\pt \rangle|\colon (\Fin_*,\vee)\ar (\Set_*,\times).$$
evaluated by taking the free commutative \emph{group} generated by $(X,\pt)$.

Finally, the symmetric algebra functor 
$$\Sym\colon (\Ab^{free,fg},\oplus)\ar (\Mod(\mb{Z})^{flat},\otimes),$$ is strong symmetric monoidal, it is filtered by
subfunctors
$\Sym_{\le n}=\oplus_{i\le n}\Sym^i \subset \Sym$
\begin{lemma}\label{lemma:sym_sp}
1. There is a natural monoidal transformation 
	$$\SP^{\infty}\ar |\mb{Z}\langle -/\pt\rangle|,$$
	of strong symmetric monoidal functors $(\Fin_*,\vee)\ar (\Set_*,\times)$.

2. In terms of the identification $(\Ab',\oplus)=(\Fin_*,\vee)$, 
there is an isomorphism $$A_n\colon \Sym'_{\le n}\simeq \mb{Z}\langle -\rangle\circ \SP^n.$$
Moreover, there are monoidal isomorphisms
$$A\colon \Sym'\simeq \mb{Z}\langle -\rangle\circ \SP^{\infty},$$
and
$$B\colon \mb{Z}\langle -\rangle'\simeq \mb{Z}\langle -\rangle\circ |\mb{Z}\langle -/\pt\rangle|,$$
of strong symmetric monoidal functors $(\Ab',\oplus)\ar (\Mod(\mb{Z})^{flat},\otimes)$.

3. There is a natural monoidal transformation $\phi'\colon \Sym'(-)\ar \mb{Z}\langle-\rangle'$ of symmetric monoidal functors 
	$(\Ab',\oplus)\ar (\Mod(\mb{Z}),\otimes)$.
\end{lemma}
\begin{proof}
1. Since $\SP^{\infty}$ is the free commutative monoid functor, by universality 
	it admits a natural transformation to
	the free commutative group functor $|\mb{Z}\langle -/\pt\rangle|$.

2. Given $S_+=S\sqcup \pt\in \Fin_*$ and $s\in S$, denote the corresponding 
	generator of the algebra $\Sym(\mb{Z}^S)$ by $e_s$ and the corresponding element in $\mb{Z}\langle \SP^{\infty}(X,\pt)\rangle$ by $s$.
Since  $\mb{Z}\langle \SP^{\infty}(X,\pt)\rangle$ is a commutative algebra with the unit $\pt$, 
there is a unique map of algebras
$$A\colon \Sym(\mb{Z}^S)\ar \mb{Z}\langle \SP^{\infty}(S_+,\pt)\rangle$$
determined by $A(e_s)=s-\pt$ for each $s\in S$. 
It is straighforward to check that this induces an isomorphisms $\Sym'(\mb{Z}^S)\ar \mb{Z}\langle \SP^{\infty}(S_+,\pt)\rangle$,
which is functorial and symmetric monoidal in the based set $S_+$.
Clearly $A$ restricted to $\Sym'_{\le n}$ gives the required isomorphism $A_n$.
The isomorphism $B\colon \mb{Z}\langle \mb{Z}^S\rangle\ar \mb{Z}\langle |\mb{Z}\langle S_+/\pt\rangle|\rangle$
follows from the identification $\mb{Z}^S\simeq \mb{Z}\langle S_+/\pt\rangle$.

3. Since $\Ab'=\Fin_*$, we obtain monoidal transformation $\Sym'\ar \mb{Z}\langle -\rangle'$.
\end{proof}

It follows that for a connected space $(X,\pt)\in \simpl{\Fin}_*$, one gets an isomorphism
\begin{equation}\label{eq:sp_chains}
	C_*(\SP^{n}(X,\pt))\simeq \mb{L}\Sym_{\le n}(\tilde{C}_*(X))\in D(\Ab),
\end{equation}
hence for any coefficients ring $R$, the homology $H_*(\SP^{\infty}(X,\pt);R)$ is a functor of the reduced chains $\tilde{C}_*(X)=C_*(X,\pt;R)\in D(\Mod(R))$. 
For instance one obtains a natural isomorphism of Pontryagin rings
$$H_*(\SP^{\infty}(X,\pt);R)\simeq \oplus_i H_*(\SP^n X,\SP^{n-1} X;R)\simeq \oplus_i H_*((X,\pt)^{\wedge n}/\Sigma_n;R),$$ 
and all these terms functorially depend on the reduced homology $H_*(X,\pt;R)$, once
an equivalence
$$\bigoplus_i H_i(X,\pt;R)[i]\simeq C_*(X,\pt)\otimes R\in D(\Mod(R)),$$ 
inducing the identity on homology,
is fixed \cite{Dold}.
Notice that the set of such equivalences is a torsor over $\prod_i \Ext^1_R(H_i(X;R),H_{i+1}(X;R))$.
In general
$H_*((X,\pt)^{\wedge n}/\Sigma_n)$
is not a functor of the abelian group $H_*(X,\pt)$ alone.

We formulate the Dold-Thom theorem as follows.
\begin{thrm}\label{thrm:Dold_Thom}
	For any connected $(X,\pt)\in \simpl{\Fin}_*$,
	the monoidal transformation
	$$\mb{L}\Sym'\ar \mb{L}\mb{Z}\langle -\rangle',$$
	evaluated on the reduced chain complex $I'(X,\pt)=N(\mb{Z}\langle X/\pt\rangle)$,
	is a quasi-isomorphism.
\end{thrm}

\begin{thrm}\label{thrm:dpt}
	There is an isomorphism 
	$$\DP\colon \mb{L}\Sym\ar[\sim] \mb{L}\mb{Z}\langle-\rangle,$$
of strong symmetric monoidal functors 
	$(D^{<-1}_{perf}(\Ab),\oplus)\ar (D(\Mod(\mb{Z})),\otimes)$.
\end{thrm}
\begin{proof}
	The morphism $\phi'\colon \Sym'\ar \mb{Z}\langle -\rangle'$ provided by Lemma \ref{lemma:sym_sp},
	according to Theorem \ref{thrm:LF_LG},
	give rise to a monoidal transformation $\phi\colon \mb{L}\Sym\ar \mb{L}\mb{Z}\langle -\rangle$
	of symmetric monoidal functors from $(D^{<-1}_{perf}(\Ab),\oplus)$ to $(D(\Mod(\mb{Z})),\otimes)$.
	By Dold-Thom theorem $\phi$ is an isomorphism.
	Let $\DP$ to be equal $\phi$.
\end{proof}
This gives a functorial version of the Dold-Puppe-Thom isomorphism \cite[Satz 4.16]{DP}. Namely
by \ref{prop:chains_em} there is a monoidal isomorphism
\begin{equation}\label{eq:em_chains}
	C^{sing}_*(\EM(-))\simeq \mb{L}\Sym(-),
\end{equation}
of strong symmetric monoidal functors from $(D^{<-1}_{perf}(\Ab),\oplus)$ to $(D(\Mod(\mb{Z})),\otimes)$.
\begin{cor}\label{cor:dtp}
	For any coefficient ring $R$, 
the Dold-Puppe-Thom isomorphism $$H_*(K(A,n);R)\simeq \mb{L}_*\Sym(A[n])\os{L}{\otimes} R$$ is functorial in the abelian group $A\in \Ab^{fg}$ for $n>1$.
\end{cor}
It follows that there is a natural grading 
$$H_*(K(A,n);R)\simeq \oplus_d H_*(K(A,n);R)^d$$ with pieces 
$H_*(K(A,n);R)^d\simeq \mb{L}_*\Sym^d(A[n])\os{L}{\otimes} R.$
Moreover this functorial  grading is with respect to the Pontryagin product in $H_*(K(A,n);R)$.

\begin{rmrk}
	The construction of $\DP$ in \cite[Satz 4.16]{DP} also involves Moore space, the above result shows that the construction
	is functorial on the level of homotopy categories.
	As pointed out in the introduction, it fails for $n=1$ and torison group $A$ \cite{Mikhailov}.
	In the case $A\in \Ab^{free,fg}$ the statement was established in \cite{Touze}.

Let us mention that there is no generalization of the splitting, provided by the Dold-Puppe-Thom isomorphism, to the level of $\infty$-categorical localization of 
	chain complexes:
	there are examples of toric fibrations such that the corresponding Lerray-Serre spectral sequence has non-trivial differentials at $E_2$, see e.g. \cite{Totaro}.
	On the other hand, a similar construction can be used to establish a version of the isomorphism in the case of diagrams of free (shifted) abelian groups 
	provided with
	a lift to a diagram of free associative groups \cite[section 4]{GZ}.
	One can say more if certain primes in the coefficients are invertible \cite{Solomadin}.
\end{rmrk}

\section{Binomial algebras}\label{sect:bin}
\begin{dfn}
	For $V\in \Ab^{free,fg}$,
	the free algebra of divided powers $\Gamma(V)\subset \Sym(V)\otimes_{\mb{Z}}\mb{Q}$
	is the algebra generated by expressions $v^{[n]}:=\frac{v^n}{n!}$ for all $n\ge 0$, $v\in V$.
\end{dfn}
\begin{dfn}
	For $V\in \Ab^{free,fg}$
	The free binomial algebra $\Bin(V)$ is the algebra of all
	integer-valued polynomial functions on $V^\vee$.
\end{dfn}
By naturality this defines a strong symmetric monoidal functor
$\Bin(-)\colon (\Ab^{free,fg},\oplus)\ar (\Ab^{flat},\otimes)$.
For example $\Bin(\mb{Z})$ is the $\mb{Z}$-span of the binomials $\binom{x}{n}\in \mb{Q}[x]$ of degree $n$.
The comultiplication is given by:
$$\Delta\binom{x}{n}=\binom{x_1+x_2}{n}=\sum\limits_{i+j=n} \binom{x_1}{i}\binom{x_2}{j}.$$
It follows that the map $\binom{x}{n}\to x^{[n]}$ gives an isomorphism of 
coalgebras $\Bin(\mb{Z})\simeq \Gamma(\mb{Z})$.
We denote by $\Bin^{\le k}(-)$ and $\Gamma^{\le k}(-)$ the multiplicative filtrations corresponding to 
the degree of polynomial functions.
Let $\Gr \Bin(-)$ denotes the functor obtained by passing to the associated graded with respect to the degree filtration.
Using \ref{prop:groups_as_functors} it easy to see the following.
\begin{prop}
	There are isomorphisms of strong symmetric monoidal functors:
	$$\Gr \Bin\simeq \Gamma\colon (\Ab^{free,fg},\oplus)\lar (\Mod(\mb{Z})^{flat},\otimes),$$
	and
	$$\Bin''\simeq \Gamma''\colon (\Ab'',\oplus)\lar (\Mod(\mb{Z})^{flat},\otimes).$$
\end{prop}

Apart from the functorial \emph{increasing} filtration 
$$\Bin^{\le n}(\mb{Z}^S)=\mb{Z}\langle\prod\limits_{s\in S} \binom{e_s}{n_s}\mid \sum\limits_{s\in S} n_s\le n\rangle,$$ we have 
the \emph{decreasing} filtration $\IBin''_{>n}\subset \Bin''$ defined as follows:
$$\IBin''_{>n}(\mb{Z}^S)=\mb{Z}\langle\prod\limits_{s\in S} \binom{e_s}{n_s}\mid \sum\limits_{s\in S} n_s>n\rangle.$$
This filtration is not functorial in general: the map $f\colon \mb{Z}^{S_1}\ar \mb{Z}^{S_2}$
induces a map $\IBin''_{>n}(\mb{Z}^{S_1})\ar \IBin''_{>n}(\mb{Z}^{S_2})$ only if $f\in \Ab''$.

Let $\BiAlg$ denotes the category of co/commutative bialgebras in $(\Mod(\mb{Z}),\otimes)$.
We start by noting that the isomorphism in part 3 of Lemma \ref{lemma:sym_sp} is a morphism of bialgebras.
\begin{dfn}\label{dfn:bin_dual}
Define the bialgebra $\Sym'(\mb{Z})$ as follows:
	\begin{enumerate}
		\item The multiplication is given by the identification  $\Sym'(\mb{Z})=\mb{Z}[t]$.
		\item The comultiplication is given by $\Delta t=t_1+t_2+t_1t_2=t\otimes 1+1\otimes t+t\otimes t$.
		\item The unit and the augmentation ideal are given by $1\in \mb{Z}[t]$ and $t\mb{Z}[t]$ respectively.
	\end{enumerate}
\end{dfn}
In other words $\Spec \Sym'(\mb{Z})=(\mb{A}^1,\times)$ is a monoid scheme with coordinate $t$ at $1\in \mb{A}^1$ 
(the bialgebra  $\Sym'(\mb{Z})$ lacks an antipod).
We extend the construction to the functor 
$$\Sym'\colon \Ab'\ar \BiAlg,$$
by $\Sym'(\mb{Z}^S)=\otimes_{s\in S}\Sym'(\mb{Z})$.
Similarly, we have group algebra
functor $\mb{Z}\langle -\rangle\colon \Ab\ar \BiAlg$.
\begin{lemma}
	The natural transformation $\phi'\colon \Sym'\ar \mb{Z}\langle-\rangle'$ in part 3 of \ref{lemma:sym_sp}, is
	a morphism of functorial bialgebras.
\end{lemma}
\begin{proof}
	For a moment, denote a generator in $\Sym'(\mb{Z}^S)$, corresponding to $s\in S$,  by $e_s\in S$ and the corresponding element in $\mb{Z}^S$ by $v_s$.
	By the construction of $\phi'_S\colon \Sym'(\mb{Z}^S)\ar \mb{Z}\langle \mb{Z}^S\rangle$, $\phi'_S(e_s)=v_s-1$, 
	where $1$ is the unit of the group algebra $\mb{Z}\langle \mb{Z}^S\rangle$ corresponding to $0\in \mb{Z}^S$.
	By the definition of the comultiplication in the group algebra $\mb{Z}\langle \mb{Z}^S\rangle$ we have 
	$$\Delta(v_s-1)=v_s\otimes v_s-1\otimes 1=(v_s-1)\otimes (v_s-1)+(v_s-1)\otimes 1 + 1\otimes (v_s-1).$$
	On the $\Sym'$'side, the comultiplication is given by $\Delta(e_s)=e_s\otimes e_s+e_s\otimes 1+1\otimes e_s$.
	Hence $\phi'_S$ commutes with $\Delta$ and is a morphism of bialgebras.
\end{proof}

Consider the restriction $\Bin''\colon \Ab''\ar \BiAlg$.
Since $\Bin''(\mb{Z}^S)\in \BiAlg$ has a preferred unordered basis given by binomials $\prod_{s\in S} \binom{e_{s}}{n_s}$,
one can define the dual $\Bin''(\mb{Z}^S)^\vee$ of the same cardinality as $\Bin''(\mb{Z}^S)$.
Set $\Bin''^{\vee}(V):=(\Bin''(V^\vee))^\vee$, this defines a functorial bialgebra
$\Bin''^{\vee}\colon \Ab'\ar \BiAlg$.
\begin{prop}\label{prop:bin_dual_sym}
	We have an isomorphism $\Bin''^\vee\simeq \Sym'$ of functors
	$\Ab'\ar \BiAlg$.
Under this isomorphism we have ${\IBin''_{>n}}^{\bot} \simeq \Sym'_{\le n}$
	and ${\Bin''_{\le n}}^{\bot}\simeq \Sym'_{>n}$.
\end{prop}
\begin{proof}
	We shall show an isomorphism of bialgebras $\Bin''(\mb{Z})^\vee\simeq \Sym'(\mb{Z})$.
	If we have a bialgebra $E$ with basis $e_i$, let $f^i$ denote its duals in $F:=E^\vee$.
	Set $u(e,f):=\sum e_i f^i$. It is convenient to pack 
	the multiplication $m$ and the comultiplication $\Delta$ of bialgebras $F$ and $E$ 
	into the formal indentities
	$$\Delta_e u(e,f)=m_f(u(e',f),u(e'',f))$$ and 
	$$m_e(u(e,f'),u(e,f''))=\Delta_f u(e,f),$$
	where $f'$ denotes $f\otimes 1$, $\Delta_e$ denotes the application of the comultiplication in $E$ involving only $e$-variables, and so on.
	Our case corresponds to the formal function
	$$u(x,t):=(1+t)^x=\sum\limits_n \binom{x}{n}t^n,$$
	i.e. we define the linear isomorphism $\Bin''(\mb{Z})\simeq \Sym'(\mb{Z})^{\vee}$ in terms of the basis 
	by requiring that $t^n\in \Sym'(\mb{Z})$ is dual to $\binom{x}{n}\in \Bin(\mb{Z})$.
	The isomorphism commutes with the multiplication and the comultiplication, thanks to the identities 
	$$\Delta_x (1+t)^{x}=m_t((1+t)^{x'}, (1+t)^{x''})$$ and
	$$m_x((1+t')^{x},(1+t'')^{x})=\Delta_t(1+t)^x.$$
	Since $\IBin''_{>n}(\mb{Z})$ is spanned by $\binom{x}{i},i>n$, 
	which are duals of $t^i$,
	we have $\IBin_{>n}(\mb{Z})^\bot\subset \Sym'(\mb{Z})^{\vee}$, the module of functionals vanishing on $\IBin_{>n}(\mb{Z})$, 
	is spanned by $t^j,j\le n$.
	Similarly $\Bin_{\le n}(\mb{Z})^{\bot}\subset \Sym'(\mb{Z})^{\vee}$ is spanned by $t^j,j>n$.
	
\end{proof}
In particular the product in $\Bin$ restricts to a morphism 
$$\IBin''_{>n}\otimes \IBin''_{>m}\ar \IBin''_{>\max(n,m)},$$
in fact $\IBin''_{>m}$ is a non-multiplicative decreasing filtration.
Since $\IBin''_{\ge 0}=\Bin''$, $\IBin''_{>n}$ is an ideal in $\Bin''$ for each $n$.

The completion of the monoid scheme $(\mb{A}^1,\times)$ at $1\in \mb{A}^1$ (i.e. at the ideal $(t)$), 
is equivalent to the completion of $\mb{G}_m$ at $1\in \mb{G}_m$.
Define $\widehat{\Sym'}$ as the power-series version of \ref{dfn:bin_dual}, then
$\wh{\mb{G}}_m\simeq \Spf \widehat{\Sym'}(\mb{Z}).$
It follows that the Cartier dual of the commutative group scheme $\Spec\Bin(\mb{Z})$ 
is the formal group scheme $\wh{\mb{G}}_m$ \cite{Toen}.

\section{Cochains of Eilenberg-MacLane spaces}\label{sect:cochains_em}
Binomial algebra is an algebra over the monad $\Bin\colon \Ab^{flat}\ar \Ab^{flat}$ and is obtained from our definition in a standard manner.
In other words it is commutative algebra with operations correspondings to binomials, i.e. its a lambda-ring such that the correponding Adams operations are identity 
\cite{Elliot}.
The basic example is given by the integer-valued functions on a set, in particular
the cochains of a simplical set is a cosimplicial binomial algebra.
The relation of derived binomial algebras to topology is well-known, 
for the general discussion we refer to \cite{Ekedahl},\cite{Toen},\cite{Horel},\cite{KSZ} and \cite{Suciu}.
Here we will discuss the free derived binomial algebra functor
$$\mb{L}\Bin\colon (D^{<-1}_{perf}(\Ab),\oplus)\ar (D(\Mod(\mb{Z})),\otimes),$$
and its relation to symmetric powers of spaces from an elementary point of view.
For a connected $(X,\pt)\in \simpl{\Fin}_*$
denote the reduced cochains of $(X,\pt)$
by $\mb{Z}^{(X,\pt)}=\mb{Z}\langle X/\pt\rangle^{\vee}\in \cosimpl{\Ab''}$. 
Using \ref{prop:bin_dual_sym}, the dualization of \eqref{eq:sp_chains} immediately gives
\begin{prop}\label{prop:bin_ideals}
	For connected $(X,\pt)\in \simpl{\Fin}_*$ 
	there is an isomorphism
	$$C^*(\SP^{\infty}X,\SP^n X)\simeq \IBin''_{>n}(\mb{Z}^{(\SP^n X,\pt)}),$$
	inducing an isomorphism
	$$C^*(\SP^n X)\simeq \Bin''/\IBin''_{>n}(\mb{Z}^{(\SP^n X,\pt)}),$$
	of monoids in $D(\Mod(\mb{Z}))$, both are functorial in $X$.
\end{prop}
Passing to the limit over $n$ one can show an equivalence
$$C^*(\SP^{\infty} X)\simeq \Bin(\mb{Z}^{(\SP^{\infty} X,\pt)})$$
By Dold-Thom theorem $\SP^{\infty} X$ is homotopy equivalent to $\EM(N(\mb{Z}\langle X/\pt\rangle))$.
Thus there is, \emph{a priori} non canonical, isomorphism of monoids 
$$C^*(\EM(P))\simeq \mb{L}\Bin(P^{\vee})\in D(\Mod(\mb{Z})),$$
for $P\in D^{<-1}_{perf}(\Ab)$.
In fact, the latter can be constructed in an invariant manner.
\begin{prop}
	There is a natural monoidal isomorphism
	$$C^*(\EM(-))\simeq \mb{L}\Bin(-^{\vee}),$$
	of strong symmetric monoidal functors $(D^{<-1}_{perf}(\Ab),\oplus)^{\op}\ar (D(\Mod(\mb{Z})),\otimes)$.
\end{prop}
\begin{proof}
	Let us sketch a proof.
	Given $P\in D^{<-1}_{perf}(\Ab)$, represent it by a complex of free groups $P\in \Ch^{<0}(\Ab^{free,fg})$.
	Consider the Eilenberg-MacLane space $\EM(P)=|K(P)|$.
	Since $\mb{Z}^{|K(P)|}$ is a (cosimplicial) binomial algebra, the inclusion of linear functions $K(P^{\vee})\ar \mb{Z}^{|K(P)|}$ 
	extends to a morphism of binomial algebras $\Bin(K(P^{\vee}))\ar \mb{Z}^{|K(P)|}$.
	This morphism is a quasi-isomorphism, for details see \cite{Ekedahl},\cite{Toen},\cite{KSZ}.
\end{proof}

On the other hand, using the isomorphism $\Sym^{\vee}\simeq \Gamma$, the dual form of \eqref{eq:em_chains} says that
there is a natural isomorphism
$$C^*(\EM(-))\simeq \mb{L}\Gamma(-^{\vee}),$$
of strong symmetric monoidal functors $(D^{<-1}_{perf}(\Ab),\oplus)^{\op}\ar (D(\Mod(\mb{Z})),\otimes)$.
In particular $H^*(K(A,n))\simeq \mb{L}^*\Gamma(A[-n])$ for $n>1$ admits canonical splitting.

The above considerations imply a natural isomorphism 
$$\mb{L}\Bin\simeq \mb{L}\Gamma,$$
of strong symmetric monoidal functors $(D^{<-1}_{perf}(\Ab)^{\vee},\oplus)\ar (D(\Mod(\kk)),\otimes)$.
This identification is explained also by Theorem \ref{thrm:LF_LG_cs}: the isomorphism 
$\phi''\colon \Bin''\simeq \Gamma''$ can be lifted to the isomorphism of the left derived functors as above.

We conclude by the following remark.
Given connected $(X,\pt)\in \simpl{\Fin}_*$, 
we have a natural map of topological monoids $\SP^{\infty}(X,\pt)\ar |\mb{Z}\langle X/\pt\rangle|$,
which is an equivalence by Dold-Thom theorem.
Assume $A_*$ is a chain complex with trivial differential and $f\colon A_*\simeq C_*(X,\pt)\in D(\Ab)$
is an isomorphism.
By the above we obtain a homotopy equivalence $\EM(A_*)\simeq \SP^{\infty}(X,\pt)$.
If $A_*$ consists of free abelian groups the equivalence doesn't depend on any choices, hence
any connected based space $(X,\pt)$ with prescribed homology groups equal to $A$, provides an equivalence $F\colon \SP^{\infty}(X,\pt)\ar[\sim] \EM(A_*)$.
Then $F$ induces a split filtration on $H_*(\EM(A_*))$, which is multiplicative with respect to the Pontragin product.
	The corresponding filtration on $H^*(K(A_*))$ is \emph{not} multiplicative in general.
	Moreover, the filtration provides a non-trivial functorial invariant of the based homotopy type $(X,\pt)$ with 
	prescribed homology groups: $\tilde{H}_*(X)=A_*$.
	
	An example is provided by considering the spaces $X_1=S^2\vee S^4$ and $X_2=\mb{C}P^2$ with equal homology groups.
	Let $y_2,y_4$ be generators of $H^*(X_1;\mb{Z})$ and $t$ is the generator of $H^2(X_2;\mb{Z})$.
	Let $E=K(\mb{Z}[2]\oplus \mb{Z}[4])\simeq \mb{C}P^2\times K(\mb{Z},4)$.
	Pick the fundamental classes $x_2,x_4\in H^*(E;\mb{Z})$.
	By Dold-Thom theorem we have maps $f_i\colon X_i=\SP^1 X_i\ar K(\mb{Z}[2]\oplus \mb{Z}[4])$,
	which are determined by pull-backs $f_1^*x_2=y_2,f_1^*x_4=y_4$ and
	$f_2^*x_2=t,f_2^*x_4=t^2$.
	Denote by $F^{>n}_1,F^{>n}_2$ the decreasing filtrations on $H^*(E)$ corresponding to 
	$\IBin''_{>n}(\mb{Z}\langle X_i/\pt\rangle)$.
	Then $\mb{Q}\otimes F^{>1}_1$ is an ideal generated by $x_2^2,x_4^2,x_2x_4$, while
	$\mb{Q}\otimes F^{>1}_2$ is generated by $x_2^2=x_4,x_2^3$.
	Hence $F^{>-1}_1\neq F^{>-1}_2$.

\section{Applications}\label{sect:apps}
\subsection{Cohomology of iterated classifying stacks}
Given an affine commutative group scheme $G$ one can define
the so-called (derived) iterated classifying stacks $B^n G=B(B^{n-1}G)$.
Its cohomology with coefficients in the structure sheaf, as a Hopf algebra, can be defined by the formula
$$H^*(B^n G,\mc{O}_{B^n G}):=\mb{L}^*\mc{O}^{\otimes -}_G(\mb{Z}[-n]),$$
where the functor $\mc{O}^{\otimes -}_G$ is the monoidal functor corresponding to the Hopf algebra $\mc{O}_G$ \eqref{eq:grp_sch}.
Recall that a formal group law $G$ over $\mb{Z}$ corresponds to a Hopf algebra $\mc{O}(G)$ provided with an isomorphism
of commutative pro-algebras
$\mc{O}(G)\simeq \mb{Z}[t]]$.
The Cartier dual $G^{\vee}$ is an affine commutative group scheme 
provided with an isomorphism $\mc{O}(G^{\vee})\simeq\mc{O}(\mb{G}^{\vee}_a)=\Gamma(\mb{Z})$ of augmented coalgebras.
Theorem \ref{thrm:LF_LG} asserts then, that
the cohomology of $B^n G^\vee$, as a Hopf algebra, is isomorphic to $\mb{L}^*\Gamma(\mb{Z}[-n])$.

Let us recall the d\'ecalage isomorphism \cite[4.3.2.1]{Illusie}:
\begin{thrm}
	There is a chain of canonical multiplicative isomorphisms:
	$$\mathbb{L}^{i-2d} \Sym^d(A[-n])\simeq 
		\mathbb{L}^{i-d}\Lambda^d (A[-n-1])\simeq 
			\mathbb{L}^{i}\Gamma^d(A[-n-2])$$
\end{thrm}\QEDB

In particular 
$\mathbb{L}^i\Gamma(A[-n-2])\simeq 
	\bigoplus\limits_{d}\mathbb{L}^{i-2d} \Sym^d (A[-n])$.
By our results 
$$H^i(K(\mathbb{Z},n+2);\mathbb{Z})\simeq 
	\mathbb{L}^i \Bin(\mathbb{Z}[-n-2])\simeq 
		\mathbb{L}^i \Gamma(\mathbb{Z}[-n]),$$
and we obtain:
\begin{cor}
	There is a chain of natural multiplicative isomorphisms:
	$$H^i(K(\mathbb{Z},n+2);\mathbb{Z})\simeq \bigoplus\limits_{d}\mathbb{L}^{i-2d}\Sym^d(\mathbb{Z}[-n])\simeq
	\bigoplus\limits_{d'}H^{d'}(B^{n}\mathbb{G}_{a},
		\mathcal{O}_{B^{n}\mathbb{G}_a})_{i-d'}$$
\end{cor}
The description of both parts of this isomorphism occurs in \cite[A]{KP}, it served our main motivation for this work.

\subsection{Cohomology of symmetric powers of suspension spaces}\label{sect:top_app}
By Dold theorem, for a connected based space $(X,\pt)$ there is an isomorphism of groups 
\begin{equation}\label{dold_hom}
	H_*(\SP^n X;\mb{Z})\simeq \mb{L}_* \Sym_{\le n}(C_*(X,\pt;\mb{Z})),
\end{equation}
in particular the symmetric product homology is determined, up to an isomorphism, by 
the groups $H_*(X;\mb{Z})$.
It is natural to ask if the algebra $H^*(\SP^n X;\mb{Z})$ can be 
recovered from the algebra $H^*(X;\mb{Z})$. 
Quite remarkabely, Dmitry Gugnin proved the following result.
\begin{thrm}[\cite{Gugnin}]
	The algebra $H^*(\SP^n X;\mb{Z})/\Tors$ functorially 
	depends on the algebra $H^*(X;\mb{Z})/Tors$.
\end{thrm}
Further, the work \cite{Puppe} provides a simple proof of the fact that the algebra $H^*(\SP^n X;\mb{Z})$ is
determined by the monoid $C^*(X;\mb{Z})\in D(\Mod(\mb{Z}))$ 
(or equivalently by the \emph{cohomological spectrum} of $X$ in terminology of \cite{Whitehead},\cite{Palermo}).
In general the cohomology algebra of a cartesian square $H^*(X^2;\mb{Z})$ is not determined by $H^*(X;\mb{Z})$ \cite{Palermo},
so one can expect that the ring structure of $H^*(\SP^2 X;\mb{Z})$ is not determined by $H^*(X;\mb{Z})$ alone.

Assume  $(X,\pt)\in \Top_*$ is a connected based space such that there is a map
$f\colon X\ar X\vee X$ inducing the diagonal morphism $f_*\colon \tilde{C}_*(X)\ar \tilde{C}_*(X)\oplus \tilde{C}_*(X)$ in $D(\Ab)$.
Here $\tilde{C}_*(X)$ denotes the reduced chain complex of $C_*(X,\pt)$.
For example any suspension space $X=\Sigma Y$ admits such a map.

\begin{thrm}\label{thrm:splitting_sus}
	For $(X,\pt)\in \Top_*$ and $f\colon X\ar X\vee X$ as above, 
	there is an isomorphism of algebras
	$$H^*(\SP^n X)\simeq \mb{L}^*\Gamma/\Gamma_{>n}(\tilde{C}^*(X))\simeq \oplus_{i\le n}\mb{L}^*\Gamma^{i}(\tilde{C}^*(X)),$$
	of monoids in $D(\Mod(\mb{Z}))$.
	In particular, the algebra 
	$H^*(\SP^n X;\mb{Z})$ is graded by pieces
	$$H^*(X^{\wedge i}/\Sigma_n)\simeq \mb{L}^* \Gamma^i(\tilde{C}^*(X)),i\le n,$$
	and is determined, up to an isomorphism, by $\mb{Z}$-module $\bar{H}^*(X;\mb{Z})$.
\end{thrm}
\begin{proof}
Recall that by \ref{prop:bin_ideals} the dual form of \eqref{dold_hom} provides a natural identification of ideals
$$\mb{L}^*\IBin''_{>n}(\mb{Z}\langle X/\pt\rangle^\vee)\simeq H^*(\SP^\infty X,\SP^n X;\mb{Z})\subset H^*(\SP^{\infty} X;\mb{Z}),$$
which gives  additively split (decreasing) filtration. 
One gets
an isomorphism of algebras
$$H^*(\SP^n X;\mb{Z})\simeq \mb{L}^*\Bin''/\IBin''_{>n}(\mb{Z}\langle X/\pt\rangle^\vee).$$

There is a morphism  $\Delta\colon X\ar X\vee X$ inducing $\tilde{C}_*(X)\ar \tilde{C}_*(X)\oplus \tilde{C}_*(X)$ equal to the diagonal
in $D(\Ab)$. 
Since $\Delta\in \simpl{\Ab'}$, the morphism $\Delta^\vee\colon \mb{Z}\langle X/\pt\rangle^{\oplus 2}\ar 
\mb{Z}\langle X/\pt\rangle\in \Ho(\cosimpl{\Ab''})$ is the 
 addition.
In particular the additive isomorphism 
$$\Bin''(\mb{Z}\langle X/\pt\rangle^\vee)\simeq \Gamma''(\mb{Z}\langle X/\pt\rangle^\vee)$$
	is in fact multiplicative in $\Ho(\cosimpl{\Mod(\mb{Z})})$.
Then isomorphism
$$H^*(\SP^n X;\mb{Z})\simeq 
	\mb{L}^*\Bin''/\IBin''_{>n}(\tilde{C}^*(X)),$$
translates to an isomorphism 
	$$H^*(\SP^n X;\mb{Z})\simeq \mb{L}^*\Gamma''/\IGamma''_{>n}(\tilde{C}^*(X)).$$
Though $\IBin''_{>n}$ doesn't extend to a functor from $\Ab$,
the RHS $\mb{L}^*\Gamma/\IGamma_{>n}(\tilde{C}^*(X))$ does.
So, for the suspension space $X$ we obtain an identification of algebras:
$$H^*(\SP^n X;\mb{Z})\simeq\mb{L}^*\Gamma/\IGamma_{>n}(\tilde{H}^*(X;\mb{Z})).$$
\end{proof}

As a corollary of the above proof we see that the multiplication
	$$H^*(\SP^\infty X,\SP^{n-1}X;\mb{Z})\otimes H^*(\SP^\infty X,\SP^{m-1}X;\mb{Z})\lar 
	H^*(\SP^\infty X,\SP^{\max(n,m)-1};\mb{Z}),$$
	factorizes through the inclusion $H^*(\SP^\infty X,\SP^{n+m-1}X;\mb{Z})\ari H^*(\SP^{\infty},\SP^{\max(n,m)-1};\mb{Z})$.

\printbibliography

\bigskip

\begin{flushleft}
\end{flushleft}

\end{document}